\documentclass[a4paper,DIV=10]{scrartcl}

\usepackage[T1]{fontenc}
\usepackage[latin1]{inputenc}
\usepackage[english]{babel}

\usepackage{amsthm}
\theoremstyle{plain}
\newtheorem{theorem}{Theorem}[section]
\newtheorem{proposition}[theorem]{Proposition}
\newtheorem{remark}[theorem]{Remark}
\newtheorem{definition}[theorem]{Definition}
\newtheorem{example}[theorem]{Example}
\newtheorem{corollary}[theorem]{Corollary}

\usepackage{mathptmx}
\usepackage{helvet}
\usepackage{courier}
\usepackage{makeidx}
\usepackage{graphicx}
\usepackage{multicol}
\usepackage[bottom]{footmisc}
\makeindex

\usepackage{amsmath}
\usepackage{amsfonts}
\usepackage{amssymb}
\usepackage{url}
\usepackage{xspace}
\usepackage{relsize}
\usepackage[innercaption]{sidecap}
\usepackage[font=small,labelfont=bf]{caption}
\usepackage{graphicx}

\usepackage[capitalize]{cleveref}
\crefname{equation}{}{}

\usepackage{titlesec}
\newcommand{\periodafter}[1]{#1.}
\titleformat{\section}[runin]{\normalfont\bfseries}{\textbf{\thesection.}}{1ex}{\periodafter}
\titleformat{\subsection}[runin]{\normalfont\bfseries}{\textbf{\thesubsection.}}{1ex}{\periodafter}
\titleformat{\subsubsection}[runin]{\normalfont\bfseries}{\textbf{\thesubsubsection.}}{1ex}{\periodafter}
\titleformat{\paragraph}[runin]{\normalfont\itshape}{}{1ex}{\periodafter}

\newcommand{\q}{\tilde{q}}
\newcommand{\SP}[2]{(#1, #2)}
\newcommand{\DP}[2]{\langle #1, #2\rangle}
\newcommand{\Ell}{\mathcal{E}}

\newcommand{\R}{\mathbb{R}}
\newcommand{\N}{\mathbb{N}}
\newcommand{\hnS}[1]{\hspace{-#1pt}}
\newcommand{\mean}[1]{\ensuremath{\left\{\hnS{2.5}\left\{#1\right\}\hnS{2.5}\right\}}}
\newcommand{\jump}[1]{\ensuremath{\left[\hnS{1.25}\left[#1\right]\hnS{1.25}\right]}}
\newcommand{\norm}[1]{\ensuremath{{\left|\hnS{1.25}\left|#1\right|\hnS{1.25}\right|}}}
\newcommand{\Params}{\mathcal{P}}
\newcommand{\eps}{\varepsilon}
\newcommand{\grid}{\tau_h}
\newcommand{\Grid}{\mathcal{T}_H}
\newcommand{\faces}{\mathcal{F}_h}
\newcommand{\Faces}{\mathcal{F}_H}
\newcommand{\restrict}[2]{{\ensuremath{\left. #1\right|_{#2}}}}
\newcommand{\gradient}[1]{{\nabla\hnS{1.25} #1}}

\newcommand{\Pk}{\mathbb{P}}
\newcommand{\red}{\ensuremath{\textnormal{red}}}
\newcommand{\DUNE}{\texttt{DUNE}\xspace}
\newcommand{\pymor}{\texttt{pyMOR}\xspace}
\newcommand\Cpp{C\nolinebreak[4]\hspace{-.05em}\raisebox{.4ex}{\relsize{-3}{\textbf{++}}}\xspace}

\makeatletter
\newcommand{\enorm}{\@ifstar\@enorms\@enorm}
\newcommand{\@enorms}[1]{%
  \left|\mkern-1.5mu\left|\mkern-1.5mu\left|
   #1
  \right|\mkern-1.5mu\right|\mkern-1.5mu\right|
}
\newcommand{\@enorm}[2][]{%
  \mathopen{#1|\mkern-1.5mu#1|\mkern-1.5mu#1|}
  #2
  \mathclose{#1|\mkern-1.5mu#1|\mkern-1.5mu#1|}
}
\makeatother

\begin{document}

\date{}
\title{True Error Control for the Localized Reduced Basis Method for Parabolic Problems\footnotemark[4]}

\author{Mario Ohlberger\footnotemark[5] \and Stephan Rave\footnotemark[5] \and Felix Schindler\footnotemark[5]}

\maketitle

\renewcommand{\thefootnote}{\fnsymbol{footnote}}
\footnotetext[4]{This work has been supported by the German Federal Ministry of Education and Research (BMBF) under contract number 05M13PMA.}
\footnotetext[5]{Institute for Computational and Applied Mathematics, University of M\"unster, Einsteinstrasse 62, 48149 M\"unster, Germany, \url{{mario.ohlberger, stephan.rave, felix.schindler}@uni-muenster.de}}

\begin{small}\textbf{Abstract.}
  We present an abstract framework for a posteriori error estimation for approximations of scalar parabolic evolution equations, based on elliptic reconstruction techniques \cite{MakridakisNochetto2003,LakkisMakridakis2006,DemlowLakkisEtAl2009,GeorgoulisLakkisEtAl2011}.
  In addition to its original application (to derive error estimates on the discretization error), we extend the scope of this framework to derive offline/online decomposable a posteriori estimates on the model reduction error in the context of Reduced Basis (RB) methods.
  In addition, we present offline/online decomposable a posteriori error estimates on the full approximation error (including discretization as well as model reduction error) in the context of the localized RB method \cite{OS2015}.
  Hence, this work generalizes the localized RB method with true error certification to parabolic problems.
  Numerical experiments are given to demonstrate the applicability of the approach.
\end{small}

\section{Introduction}
\label{sec:introduction}

We are interested in efficient and certified numerical approximations of pa\-ra\-bo\-lic parametric problems, such as: given a Gelfand triple of suitable Hilbert spaces $Q \subset H \subset Q'$, an end time $T_\text{end} > 0$, initial data $p_0 \in Q$ and right hand side $f \in H$, for a parameter $\mu \in \Params$ find $p(\cdot; \mu) \in L^2(0, T_\text{end}; Q)$ with $\partial_t p(\cdot; \mu) \in L^2(0, T_\text{end}; Q')$, such that $p(0; \mu) = p_0$ and
\begin{align}
        \DP{\partial_t p(t; \mu)}{q} + b\big(p(t; \mu), q; \mu\big) &= (f, q)_H &&\text{for all } q \in Q,
        \label{eq:intro:continuous_problem}
\end{align}
where $\Params \subset \R^\rho$ for $\rho \in \N$ denotes the set of admissible parameters and $b$ denotes a parametric elliptic bilinear form (see Section \ref{sec:general_framework} for details).

We consider grid-based approximations $p_h(\mu) \in Q_h$ of $p(\mu) \in Q$, obtained by formulating \eqref{eq:intro:continuous_problem} in terms of a discrete approximation space $Q_h \subset H$ (think of Finite Elements or Finite Volumes) where $b$ is replaced by a discrete counterpart acting on $Q_h$ (e.g.~in case of nonconforming approximations).

Efficiency of such an approximation for a single parameter is usually associated with minimal computational effort, obtained by adaptive grid refinement using localizable and reliable error estimates (see \cite{Ver2013} and the references therein).
For parametric problems, however, where one is interested in approximating \eqref{eq:intro:continuous_problem} for many parameters, efficiency is related to an overall computational cost that is minimal compared to the combined cost of separate approximations for each parameter.
To this end one employs model reduction with reduced basis (RB) methods, where one usually considers a common approximation space $Q_h$ for all parameters (with the notable exceptions \cite{ASU2014,Yan2015}) and where one iteratively builds a reduced approximation space $Q_\red \subset Q_h$ by an adaptive greedy search, the purpose of which is to capture the manifold of solutions of \eqref{eq:intro:continuous_problem}: $\{ p(t; \mu) \in Q \,|\, t \in [0, T_\text{end}], \mu \in \Params \}$; we refer to the monographs \cite{PR2006,HesthavenRozzaEtAl2016,QuarteroniManzoniEtAl2016} and the references therein.
One obtains a reduced problem by Galerkin projection of all quantities onto $Q_\red$ and, given a suitable parametrization of the problem, the assembly of the reduced problem allows for an offline/online decomposition such that a reduced solution $p_\red(\mu) \in Q_\red$ for a parameter $\mu \in \Params$ can be efficiently computed with a computational effort independent of the dimension of $Q_h$.
To assess the quality of the reduced solution and to steer the greedy basis generation, RB methods traditionally rely on residual based a posterior error estimates on the \emph{model reduction error} $e_\red(\mu) := p_h(\mu) - p_\red(\mu)$, $\norm{e_\red(\mu)} \leq \eta_\red(\mu)$, with the drawback that usually no information on the \emph{discretization error} $e_h(\mu) := p(\mu) - p_h(\mu)$ is available during the online phase of the computation.

In contrast, we are interested in approximations of \eqref{eq:intro:continuous_problem} which are efficient in the parametric sense as well as certified in the sense that we have access to an efficiently computable estimate on the \emph{full approximation error} $e_{h, \red}(\mu) := p(\mu) - p_\red(\mu)$, including the discretization as well as the model reduction error: $\norm{e_{h, \red}(\mu)} \leq \eta_{h, \red}(\mu)$.

For elliptic problems such an estimate is available for the localized RB multiscale method (LRBMS) \cite{AHKO2012,OS2015}, the idea of which is to couple spatially localized reduced bases associated with subdomains of the physical domain.
In addition to computational benefits due to the localization, the LRBMS also allows to adaptively enrich the local reduced bases online by solving local corrector problems, given a localizable error estimate.
Apart from the LRBMS \cite{OS2014}, we are only aware of \cite{ASU2014} and \cite{Yan2015}, where the full approximation error is taken into account in the context of RB methods.

In an instationary setting, localized RB methods were first applied in the context of two-phase flow in porous media \cite{KFH+2015} and to parabolic problems such as \eqref{eq:intro:continuous_problem} in the context of Lithium-Ion Battery simulations \cite{ORS2016}, yet in either case without error control.
In contrast, here we present the fully certified localized RB method for parabolic problems by equipping it with suitable a posteriori error estimates.
As argued above, it is beneficial to have access to several error estimates which can be evaluated efficiently during the online phase: for instance to later enable online adaptive basis enrichment, one could (i) solve local corrector problems, given $\eta_\red$, whenever the reduced space is not rich enough; or one could (ii) locally adapt the grid, given $\eta_{h, \red}$, whenever $Q_h$ is insufficient.

Therefore, we present a general framework for a posteriori error estimation for parabolic problems, which will enable us to obtain either of the above estimates.
It is based on the elliptic reconstruction technique, introduced for several discretizations and norms in \cite{MakridakisNochetto2003,LakkisMakridakis2006,DemlowLakkisEtAl2009,GeorgoulisLakkisEtAl2011}. In this contribution we reformulate this approach in an abstract setting, allowing for a novel application in the context of RB methods.
In particular, this technique allows to reuse existing a posteriori error estimates for elliptic diffusion problems.

\section{General Framework for A Posteriori Error Estimates}
\label{sec:general_framework}

In the following presentation we mainly follow \cite{GeorgoulisLakkisEtAl2011}, reformulating it in an abstract Hilbert
space setting and slightly extending it by allowing non-symmetric bilinear forms.
We drop the parameter dependency in this section to simplify the notation.

\begin{definition}[Abstract parabolic problem]
Let $Q$ be a Hilbert space, densely embedded in another Hilbert space $H$ (possibly $Q = H$),
and let $\widetilde{Q} \subseteq H$ be a finite dimensional approximation space for $Q$,
not necessarily contained in $Q$.
Denote by $\SP{\cdot}{\cdot}$, $\|\cdot\|$ the $H$-inner product and the norm induced by it.

Let $f \in H$, and let $b: (Q + \widetilde{Q}) \times (Q + \widetilde{Q}) \to \R$ be a bilinear
form which is continuous and coercive on $Q$.
Let further $\enorm{\cdot}$ be a norm over $Q + \widetilde{Q}$, which coincides with the square root of the symmetric part of $b$ over $Q$.

Our goal is to bound the error $e(t) := p(t) - \tilde{p}(t)$ between the true (analytical) solution $p \in L^2(0, T_\text{end}; Q)$,
$\partial_t p \in L^2(0, T_\text{end}; Q^\prime)$ of \eqref{eq:intro:continuous_problem},
where the duality pairing $\DP{\partial_t p(t)}{q}$ is induced by the $H$-scalar product
via the Gelfand triple $Q \subseteq H = H^\prime \subseteq Q^\prime$,
and the $\widetilde{Q}$-Galerkin approximation $\tilde{p} \in L^2(0, T_\text{end}, \widetilde{Q})$, $\partial_t \tilde{p} \in L^2(0, T_\text{end}, \widetilde{Q})$, solution of
\begin{equation}\label{eq:reduced_problem}
	\SP{\partial_t \tilde{p}(t)}{ \tilde{q}} + b(\tilde{p}(t), \tilde{q}) = \SP{f}{ \tilde{q}} \qquad\text{for all } \tilde{q} \in \widetilde{Q}.
\end{equation}
\end{definition}

\begin{definition}[Elliptic reconstruction]
  \label{def:elliptic_reconstruction}
	Denote by $\widetilde{\Pi}$ the $H$-orthogonal projection onto $\widetilde{Q}$.
	For $\tilde{q} \in \widetilde{Q}$, define the elliptic reconstruction $\Ell(\tilde{q}) \in Q$ of $\tilde{q}$ to be the unique solution of the
	variational problem
	\begin{equation}\label{eq:elliptic_reconstruction}
		b(\Ell(\tilde{q}), q^\prime) = \SP{B(\tilde{q}) - \widetilde{\Pi}(f) + f}{q^\prime} \qquad\text{for all } q^\prime \in Q,
	\end{equation}
	where $B(\q) \in \widetilde{Q}$ is the $H$-inner product Riesz representative  of the functional $b(\q, \cdot)$, i.e.,
	$\SP{B(\q)}{\tilde{q}^\prime}$ $= b(\tilde{q}, \tilde{q}^\prime)$ for all $\tilde{q}^\prime \in \widetilde{Q}$.
	Note that $\Ell(\tilde{q})$ is well-defined, due to the coercivity of $b$ on $Q$.
\end{definition}

The following central property of the elliptic reconstruction follows immediately from its definition:

\begin{proposition}\label{thm:elliptic_reconstruction}
	$\tilde{q}$ is the $\widetilde{Q}$-Galerkin approximation of the solution $\Ell(\tilde{q})$ of the weak problem
	\cref{eq:elliptic_reconstruction} in the sense that $\tilde{q}$ satisfies
	\begin{equation*}
		b(\tilde{q}, \tilde{q}^\prime) = \SP{\tilde{w} - \widetilde{\Pi}(f) + f}{\tilde{q}^\prime} \qquad\text{for all } \tilde{q}^\prime \in \widetilde{Q}.
	\end{equation*}
\end{proposition}

Assume that for each $t$ we have a decomposition $\tilde{p}(t) =: \tilde{p}^c(t) + \tilde{p}^d(t)$ (not necessarily unique) where $\tilde{p}^c(t) \in Q$,
$\tilde{p}^d(t) \in \widetilde{Q}$ are the conforming and non-conforming parts of $\tilde{p}(t)$.
	We consider the following error quantities:
	\begin{align*}
		\rho(t)&:=p(t) - \Ell(\tilde{p}(t)), & \varepsilon(t) &:= \Ell(\tilde{p}(t)) - \tilde{p}(t), \\
		e^c(t) &:= p(t) - \tilde{p}^c(t), & \varepsilon^c(t)&:= \Ell(\tilde{p}(t)) - \tilde{p}^c(t).
	\end{align*}

\begin{theorem}[Abstract semi-discrete error estimate]\label{thm:abstract_estimate}
Let $C:=(2\enorm{b}^2 + 1)^{1/2}$, where $\enorm{b}$ denotes the continuity constant of $b$ on $Q$ w.r.t.~$\enorm{\cdot}$, then
\begin{align*}\label{eq:abstract_estimate}
  \|e\|_{L^2(0, T_\text{end}; \enorm{\cdot})}\leq &\|e^c(0)\| + \sqrt{3} \|\partial_t \tilde{p}^d\|_{L^2(0, T_\text{end}; \enorm{\cdot}_{Q,-1})} \\
  & + (C+1)\cdot\|\varepsilon\|_{L^2(0, T_\text{end}; \enorm{\cdot})} + C\cdot\|\tilde{p}^d\|_{L^2(0, T_\text{end}; \enorm{\cdot})}.
\end{align*}
\end{theorem}
\begin{proof}[cf.\ \cite{GeorgoulisLakkisEtAl2011}]
For each $q \in Q$, we have the error identity
\begin{equation}
	\DP{\partial_t e(t)}{q} + b(\rho(t), q) = 0,
        \label{eq:error_identity}
\end{equation}	
using the definition of $\rho$, the properties of the elliptic reconstruction and the fact, that $p$ solves \eqref{eq:intro:continuous_problem}.
Testing with $e^c(t)$ and applying Young's inequality then yields
\begin{equation}\label{eq:proof_abstract_estimate_rho_inequality}
	\partial_t \|e^c(t)\|^2  + \enorm{\rho(t)}^2 \leq 3 \enorm{\partial_t \tilde{p}^d(t)}_{Q,-1}^2 + (2\enorm{b}^2 +
	1)\cdot\enorm{\varepsilon^c(t)}^2.
\end{equation}
Hence, the claim follows by integrating \cref{eq:proof_abstract_estimate_rho_inequality} from $0$ to $T_\text{end}$ and using the triangle
inequalities $\enorm{e(t)} \leq \enorm{\rho(t)} + \enorm{\varepsilon(t)}$ and $\enorm{\varepsilon^c(t)} \leq \enorm{\varepsilon(t)} +
\enorm{\tilde{p}^d(t)}$.
\end{proof}

\begin{remark}
\label{remark:elliptic_estimate}
	According to Proposition~\ref{thm:elliptic_reconstruction}, the term $\enorm{\varepsilon(t)}$ can be bounded using any
	available a posteriori error estimator for the elliptic equation \cref{eq:elliptic_reconstruction}.
  The term $\enorm{\partial_t\tilde{p}^d(t)}_{Q,-1}$ can be bounded by $C_{H, Q}^b\|\partial_t\tilde{p}^d(t)\|$ using
  the Cauchy-Schwarz inequality, where $C_{H, Q}^b$ is a constant such that $\|q\| \leq C_{H, Q}^b\enorm{q}$ for all $q
  \in Q$.
\end{remark}

It is straightforward to modify the estimate in Theorem \ref{thm:abstract_estimate} for semi-discrete solutions $\tilde{p}(t)$ to
take the time discretization error into account:

\begin{corollary}\label{thm:fully_discrete_estimate}
  Let $\tilde{p} \in L^2(0, T_\text{end}, \widetilde{Q})$, $\partial_t \tilde{p} \in L^2(0, T_\text{end}, \widetilde{Q})$ be an arbitrary discrete approximation of
	$p(t)$, not necessarily satisfying \cref{eq:reduced_problem}.
	Let $\mathcal{R}_T[\tilde{p}](t) \in \widetilde{Q}$ denote the $\widetilde{Q}$-Riesz representative w.r.t.~the $H$-inner product
	of the time-stepping residual of $\tilde{p}(t)$, i.e.
	\begin{equation*}
		(\mathcal{R}_T[\tilde{p}](t), \tilde{q}) = \SP{\partial_t \tilde{p}(t)}{\tilde{q}} + b(\tilde{p}(t), \tilde{q}) - \SP{f}{\tilde{q}} \qquad\forall \tilde{q}
		\in \widetilde{Q}.
	\end{equation*}
	Then, with $C:=(3\enorm{b}^2 + 2)^{1/2}$, the following error estimate holds:
	\begin{equation}\label{eq:fully_discrete_estimate}
		\begin{aligned}
      \|e\|_{L^2(0, T_\text{end}; \enorm{\cdot})}
      & \leq \|e^c(0)\| + 2 \|\partial_t \tilde{p}^d\|_{L^2(0, T_\text{end}; \enorm{\cdot}_{Q,-1})} \\
      & \qquad + (C+1)\cdot\|\varepsilon\|_{L^2(0, T_\text{end}; \enorm{\cdot})} + C\cdot\|\tilde{p}^d\|_{L^2(0, T_\text{end}; \enorm{\cdot})} \\
      & \qquad + 2C_{H,Q}^b\cdot\|\mathcal{R}_T[\tilde{p}]\|_{L^2(0, T_\text{end}; H)}.
		\end{aligned}
	\end{equation}
\end{corollary}
\begin{proof}
	Since \cref{eq:error_identity} no longer holds, we gain $\mathcal{R}_T[\tilde{p}](t)$ as
	an additional source term in the error equation:
	\begin{equation*}
		\DP{\partial_t e(t)}{q} + b(\rho(t), q)  = \SP{-\mathcal{R}_T[\tilde{p}](t)}{q}.
	\end{equation*}
	The statement follows using the same line of argument as in the proof of Theorem \ref{thm:abstract_estimate},
	taking the additional term into account.
\end{proof}

\begin{example}[Implicit Euler time stepping]
        \label{example:implicit_euler}
        Let $n_t \in \mathbb{N}$ be the number of time steps and $\Delta_t :=T_\text{end} / n_t$ the (fixed) time step size.
	Let $\tilde{p}(t)$ be the $\widetilde{Q}$-valued piecewise linear function with supporting points $\tilde{p}(n \cdot \Delta_t )=:\tilde{p}^n$,
	$n=0, \ldots n_t$, such that $\tilde{p}^0 := p(0)$ and $\tilde{p}^n$ is defined for $n > 0$ as the solution of
	\begin{equation*}
                \Big(\frac{\tilde{p}^{n}- \tilde{p}^{n-1}}{\Delta_t }, \tilde{q}\Big) + b(\tilde{p}^{n}, \tilde{q}) = \SP{f}{\tilde{q}} \qquad \forall \tilde{q} \in
		\widetilde{Q}.
	\end{equation*}
	We then have for $(n-1)\cdot \Delta_t  \leq t \leq n \cdot t$ the equality
	\begin{equation*}
		\mathcal{R}_T[\tilde{p}](t) = \frac{n \cdot \Delta_t  - t}{\Delta_t } B(\tilde{p}^n - \tilde{p}^{n-1}).
	\end{equation*}
	Thus,
	\begin{equation*}
    \|\mathcal{R}_T(\tilde{p})\|_{L^2(0, T_\text{end}; H)} = \left\{ \sum_{n=1}^{n_t} \frac{\Delta_t }{3} \|B(\tilde{p}^n - \tilde{p}^{n-1})\|^2
			\right\}^{1/2}.
	\end{equation*}
	Similarly, we obtain for the other quantities in \cref{eq:fully_discrete_estimate} the bounds
	\begin{align*}
    \|\varepsilon\|_{L^2(0, T_\text{end}; \enorm{\cdot})} &\leq 2 \left\{ \sum_{n=0}^{n_t} \frac{\Delta_t }{3}
		\enorm{\varepsilon^n}^2 \right\}^{1/2}, \\
    \|\tilde{p}^d\|_{L^2(0, T_\text{end}; \enorm{\cdot})} &\leq 2 \left\{ \sum_{n=1}^{n_t} \frac{\Delta_t }{3}
	\enorm{\tilde{p}^{d,n}}^2 \right\}^{1/2} ,\\
  \|\partial_t \tilde{p}^d\|_{L^2(0, T_\text{end}; \enorm{\cdot}_{Q,-1})} &\leq \left\{ \sum_{n=1}^{n_t} \frac{1}{\Delta_t }
\enorm{\tilde{p}^{d,n} - \tilde{p}^{d,n-1}}_{Q,-1}^2 \right\}^{1/2},
	\end{align*}
	where $\varepsilon^n:=\varepsilon(n \cdot \Delta_t )$, $\tilde{p}^{d,n}:=\tilde{p}^d(n \cdot \Delta_t )$, $0 \leq n \leq n_t$.
\end{example}

\begin{example}[Reduced basis approximation]
	We can directly apply Corollary \ref{thm:fully_discrete_estimate} to obtain a posteriori estimates for standard
	reduced basis schemes.
	In this case, $Q = H$ will be some discrete \lq truth\rq\ space and $\widetilde{Q} \subseteq Q$ the reduced approximation
	space.
	The $Q$ and $H$-norms might be, in case of a conforming approximation, the $H^1_0(\Omega)$ and $L^2(\Omega)$ norms for
	some domain $\Omega$.
	\Cref{eq:fully_discrete_estimate} then reduces to
	\begin{align*}
    \|e\|_{L^2(0, T_\text{end}; \enorm{\cdot})}
    \leq \|e(0)\| + (C+1)\cdot\|\varepsilon\|_{L^2(0, T_\text{end}; \enorm{\cdot})}+ 2C_{H,Q}^b\cdot\|\mathcal{R}_T(\tilde{p})\|_{L^2(0, T_\text{end}; H)}.
	\end{align*}
  The elliptic error $\enorm{\varepsilon}_{L^2(0, T_\text{end}; Q)}$ could be bounded using a standard residual-based error
	estimator for \cref{eq:elliptic_reconstruction}.
	For parametric problems with affine parameter dependency, all appearing terms are easily offline/online
	decomposed.
\end{example}

\section{Localized reduced basis methods}
\label{section:lrbms}

We now return to the definition of the localized RB method for parabolic problems as follows.

\noindent\textbf{The continuous problem.}
Let $\Omega \subset \R^d$ for $d = 1, 2, 3$ denote a bounded connected domain with polygonal boundary $\partial\Omega$ and, following the notation of Section~\ref{sec:introduction}, let $H = L^2(\Omega)$ and $Q = H^1_0(\Omega)$.
We consider problem \eqref{eq:intro:continuous_problem} with the parametric bilinear form $b$, defined over $Q$, as
\begin{align}
    b(p, q; \mu) = \int_\Omega (\lambda(\mu)\kappa_\eps\gradient{p})\cdot\gradient{q} &&\text{for } p, q \in H^1(\Omega), \mu \in \Params,
        \label{eq:lrbms:b_and_l}
\end{align}
given data functions $\kappa_\eps \in [L^\infty(\Omega)]^{d \times d}$ and $\lambda: \mathcal{P} \to L^\infty(\Omega)$.
For $\lambda$ and $\kappa_\eps$, such that $\lambda(\mu)\kappa_\eps \in [L^\infty(\Omega)]^{d \times d}$ is bounded from below (away from 0) and above for all $\mu \in \Params$, the bilinear form $b(\cdot, \cdot; \mu)$ is continuous and coercive with respect to $Q$ for all $\mu \in \Params$.
Thus, a unique solution $p(\cdot; \mu) \in L^2(0, T_\text{end}; Q)$ of problem \eqref{eq:intro:continuous_problem} exists for all $\mu \in \Params$, if $f$ is bounded.

We continue with the definition of the discretization in order to define the approximation space $\tilde{Q}$, to extend the definition of $b$ onto $\tilde{Q}$ and to introduce the relevant norms.

\medskip

The main idea of localized RB methods is to partition the physical domain $\Omega$ into subdomains in the spirit of domain decomposition methods and to generate a local reduced basis on each subdomain, as opposed to a single reduced basis with global support.
Coupling across subdomains is achieved by a symmetric weighted interior penalty discontinuous Galerkin (SWIPDG) scheme \cite{ESZ2009} for the high-dimensional as well as the reduced discretization.

\noindent\textbf{The discretization.}
To discretize \eqref{eq:intro:continuous_problem} we require two nested partitions of $\Omega$: a coarse one, $\Grid$ with elements (subdomains) $T \in \Grid$, and a fine one, $\grid$ with elements $t \in \grid$ (note that we use $t$ to denote elements of the computational grids, not to be confused with the time $t$).
Within each subdomain $T \in \Grid$ we allow for any local approximation of $Q$ and $b$ by discrete counterparts $Q_h^{k, T}$ and $b_h^T$ of order $k \geq 1$, associated with the local grid $\grid^T := T \cap \grid \subset \grid$.
In particular we consider (i) local conforming approximations by setting $Q_h^{k, T}$ to $\{ q_h \in C^0(T) \;|\; \restrict{q_h}{t} \in \Pk_k(t)\;\; \forall t \in \grid^T \} \subset H^1(T)$ and $b_h^T$ to $\restrict{b}{T}$, where $\Pk_k(\omega)$ denotes the space of polynomials over $\omega \subseteq \Omega$ of order up to $k \in \N$; (ii) local nonconforming approximations by setting $Q_h^{k, T}$ to $\{ q_h \in L^2(T) \;|\; \restrict{q_h}{t} \in \Pk_k(t)\;\; \forall t \in \grid^T \} \subset L^2(T)$ and $b_h^T$ to the following SWIPDG bilinear form: for $p, q \in Q_h^{k, T}$ and $\mu \in \Params$, we define
\begin{align*}
        b_h^T(p, q; \mu) := b^T(p, q; \mu) + \sum_{e \in \faces^T} b_e(p, q; \mu),
\end{align*}
with $b^T(p, q; \mu) := \int_T (\lambda(\mu)\kappa_\eps\gradient{p})\cdot\gradient{q}$, where $\faces^T$ denotes the set of all inner faces of $\grid^T$ that share two elements $t^-, t^+ \in \grid^T$.
The face bilinear form $b_e$ for any inner or boundary face $e$ of $\grid$ is given by $b_e(p, q; \mu) := b_c^e(q, p; \mu) + b_c^e(p, q; \mu) + b_p^e(p, q; \mu)$ with the coupling and penalty face bilinear forms $b_c^e$ and $b_p^e$ given by
\begin{align*}
  b_c^e(p, q; \mu) := \int_e -\mean{\lambda(\mu) \kappa_\eps \tilde{\Pi} \gradient{p}}_e \jump{q}_e
        &&\text{and}&&
        b_p^e(p, q; \mu) := \int_e \sigma_e(\mu) \jump{p}_e \jump{q}_e,
\end{align*}
respectively, with the $L^2$-orthogonal projection $\tilde{\Pi}$ from Definition \ref{def:elliptic_reconstruction}.
Given a function $q$ which is two-valued on faces, its jump and weighted average are given by $\jump{q}_e := q^- - q^+$ and $\mean{q}_e := \omega_e^- q^- + \omega_e^+ q^+$, respectively, on uniquely oriented inner faces $e = t^- \cap t^+$ for $t^\pm \in \grid$, and by $\jump{q}_e := \mean{q}_e := q$ for boundary faces $e = t^- \cap \partial\Omega$, with the locally adaptive weights given by $\omega_e^- := \delta_e^+ (\delta_e^+ + \delta_e^-)^{-1}$ and $\omega_e^+ := \delta_e^- (\delta_e^+ + \delta_e^-)^{-1}$, respectively, with $\delta_e^\pm := n_e \cdot \kappa_\eps^\pm \cdot n_e$.
Here, $n_e \in \R^d$ denotes the unique normal to a face $e$ pointing away from $t^-$ and $q^\pm := \restrict{q}{t^\pm}$.
The positive penalty function is given by $\sigma_e(\mu) := \sigma h_e^{-1} \mean{\lambda(\mu)}_e \sigma_\eps^e$, where $\sigma \geq 1$ denotes a user-dependent parameter, $h_e > 0$ denotes the diameter of a face $e$, and the locally adaptive weight is given by $\sigma_\eps^e := \delta_e^+ \delta_e^- (\delta_e^+ + \delta_e^-)^{-1}$ on inner faces and by $\sigma_\eps^e := \delta_e^-$ on boundary faces.

Given local approximations $Q_h^{k, T}$ and $b_h^T$ on each subdomain $T \in \Grid$, we define the DG space by $Q_h^k := \oplus_{T \in \Grid} Q_h^{k, T}$ and couple the local discretizations along a coarse face $E \in \Faces$, by SWIPDG fluxes to obtain the global bilinear form $b: \Params \to [Q_h^k \times Q_h^k \to \R]$, by
\begin{align*}
        b(p, q; \mu) := \sum_{T \in \Grid} b_h^T(p, q; \mu) + \sum_{E \in \Faces}\sum_{e \in \faces^E} b_e(p, q; \mu),
\end{align*}
for $p, q \in Q_h^k$, $\mu \in \Params$, where $\Faces$ denotes the set of all faces of the coarse grid $\Grid$ and where $\faces^E$ denotes the set of fine faces of $\grid$ which lie on a coarse face $E \in \Faces$.
Note that $b$ is continuous and coercive with respect to $Q_h^k$ in the DG norm $\enorm{\cdot}_{\cdot}$ (see the next section) if the penalty parameter $\sigma$ is chosen large enough (see \cite{OS2015} and the references therein concerning the choice of $\sigma$).

Depending on the choice of $\Grid$ and the local approximations, the above definition covers a wide range of discretizations, ranging from a standard conforming to a standard SWIPDG one; we refer to \cite{OS2015} for details.
The semi-discrete problem for a single parameter then reads as \eqref{eq:reduced_problem} with $\tilde{Q} = Q_h^k$.
Presuming $p_0 \in Q_h^k$ and using implicit Euler time stepping (compare Example \ref{example:implicit_euler}) the fully-discrete problem reads: for each time step $n > 0$ find the DoF vector of $p_h^n(\mu) := p_h(n \cdot \Delta_t , \mu) \in Q_h^k$, denoted by $\underline{p_h^n(\mu)} \in \R^{\dim Q_h^k}$, such that
\begin{align}
        \Big( \underline{M_h} + \Delta_t \; \underline{b(\mu)} \Big) \underline{p_h^n(\mu)} = \Delta_t \; \underline{f_h} + \underline{p_h^{n - 1}(\mu)},
        \label{eq:lrbms:fully_discrete}
\end{align}
where $\underline{M_h}, \underline{b(\mu)} \in \R^{\dim Q_h^k \times \dim Q_h^k}$ and $\underline{f_h} \in \R^{\dim Q_h^k}$ denote the matrix and vector representations of $(\cdot,\cdot)_{L^2(\Omega)}$, $b(\cdot,\cdot;\mu)$ and $(f,\cdot)_{L^2(\Omega)}$, respectively, with respect to the basis of $Q_h^k$.

\noindent\textbf{Model reduction.}
Let us assume that we are already given a reduced space $Q_\red \subset Q_h^k$ (we postpone the discussion of how to find $Q_\red$ to Section \ref{sec:numerical_example}).
Given $Q_\red$, we formally arrive at the reduced problem simply by Galerkin projection of \eqref{eq:lrbms:fully_discrete} onto $Q_\red$, just like traditional RB methods: for each time step $n > 0$ find the reduced DoF vector $\underline{p_\red^n(\mu)} \in \R^{\dim Q_\red}$, such that
\begin{align}
        \Big( \underline{M_\red} + \Delta_t \; \underline{b_\red(\mu)} \Big) \underline{p_\red^n(\mu)} = \Delta_t \; \underline{f_\red} + \underline{p_\red^{n - 1}(\mu)},
        \label{eq:lrbms:reduced}
\end{align}
with $p_\red^0(\mu) := \Pi_\red(p_0)$, where $\Pi_\red$ denotes the $L^2$-orthogonal projection onto $Q_\red$, and where $\underline{M_\red}, \underline{b_\red(\mu)} \in \R^{\dim Q_\red \times \dim Q_\red}$ and $\underline{f_\red} \in \R^{\dim Q_\red}$ denote the matrix and vector representations of $(\cdot,\cdot)_{L^2(\Omega)}$, $b(\cdot,\cdot;\mu)$ and $(f,\cdot)_{L^2(\Omega)}$, respectively, with respect to the basis of $Q_\red$.

As usual with RB methods, we can achieve an efficient offline/online splitting of the computational process by precomputing the restriction of the functionals and operators arising in \eqref{eq:lrbms:fully_discrete} to $Q_\red$, if those allow for an affine decomposition with respect to the parameter $\mu$.
For standard RB methods, where $Q_\red$ is spanned by reduced basis functions with global support, the matrix representation of the reduced $L^2$-product, for instance, would be given by $\underline{M_\red} = \underline{\Pi_\red} \cdot \underline{M_h} \cdot \underline{\Pi_\red}^\perp$, where $\underline{\Pi_\red} \in \R^{\dim Q_\red \times \dim Q_h^k}$ denotes the matrix representation of $\Pi_\red$ (each row of $\underline{\Pi_\red}$ corresponds to the DoF vector of one reduced basis function).
For localized RB methods, however, we are given a local reduced basis on each subdomain $T \in \Grid$ (reflected in the structure of the reduced space, $Q_\red = \oplus_{T \in \Grid} Q_\red^T$) and all operators and functionals are localizable with respect to $\Grid$.
Thus, the reduced basis projection can be carried out locally as well.
For instance, since $\SP{p}{q}_{L^2(\Omega)} = \sum_{T \in \Grid} \SP{\restrict{p}{T}}{\restrict{q}{T}}_{L^2(T)}$, we locally obtain $\underline{M_\red^T} = \underline{\Pi_\red^T} \cdot \underline{M_h^T} \cdot \underline{\Pi_\red^T}^\perp \in \mathbb{R}^{\dim Q_\red^T \times \dim Q_\red^T}$ for all $T \in \Grid$, where $\underline{\Pi_\red^T} \in \R^{\dim Q_\red^T \times \dim Q_h^{k, T}}$ and $\underline{M_h^T} \in \R^{\dim Q_h^{k, T} \times \dim Q_h^{k, T}}$ denote the matrix representations of the local $L^2$-orthogonal reduced basis projection and $\SP{\cdot}{\cdot}_{L^2(T)}$, respectively.
The reduced $L^2$-product matrix $\underline{M_\red} \in \R^{\dim Q_\red \times \dim Q_\red}$, with $\dim Q_\red = \sum_{T \in \grid} \dim Q_\red^T$, is then assembled by combining the local matrices using a standard DG mapping with respect to $Q_\red$.
In the same manner, the reduction of $b$ can be carried out locally by projecting the local bilinear forms on each subdomain as well as the coupling bilinear forms with respect to all neighbors, yielding sparse reduced operators and products; we refer to \cite{OS2015} for details and implications.

\section{Error analysis}
\label{sec:application_to_lrbms}

For our analysis we introduce the broken Sobolev space $H^1(\grid) := \big\{ q \in L^2(\Omega) \;\big|\; \restrict{q}{t}
\in H^1(t)\;\; \forall t \in \grid \big\}$, containing $Q + \tilde{Q}$, since $Q_\red \subset Q_h^k \subset H^1(\grid) \subset L^2(\Omega)$ and $H^1(\Omega) \subset H^1(\grid)$.
Note that the domain of all operators, products and functionals of the previous section can be naturally extended to $H^1(\grid)$, for instance by using the broken gradient operator $\nabla_h$, which is locally defined by $\restrict{(\nabla_h q)}{t} := \nabla(\restrict{q}{t})$ for all $t \in \grid$.
Using said operator in the definition of $b^T$, we define the parametric energy semi-norm (which is a norm only on $H^1_0(\Omega)$) by $|q|_\mu := \big(\sum_{T \in \Grid} b^T(q, q; \mu)\big)^{1/2}$ and the parametric DG norm by $\enorm{q}_\mu := \big(\sum_{T \in \Grid} b^T(q, q; \mu) + \sum_{e \in \faces} b_p^e(q, q; \mu) \big)^{1/2}$, for $\mu \in \Params$ and $q \in H^1(\grid)$, respectively, where $\faces$ denotes the set of all faces of $\grid$.
Note that $\enorm{q}_\mu = |q|_\mu$ for $q \in H^1_0(\Omega)$.

Since we presume $\lambda$ to be affinely decomposable with respect to $\mu$, there exist $\Xi \in \N$ strictly positive coefficients $\theta_\xi: \Params \to \R$ and nonparametric components $\lambda_\xi \in L^\infty(\Omega)$, such that $\lambda(\mu) = \sum_{\xi = 1}^\Xi \theta_\xi(\mu) \lambda_\xi$.
We can thus compare $\lambda$, and in particular $\enorm{\cdot}_\cdot$, for two parameters by means of $\alpha(\mu, \overline{\mu}) := \min_{\xi = 1}^\Xi \theta_\xi(\mu) \theta_\xi(\overline{\mu})^{-1}$ and $\gamma(\mu, \overline{\mu}) := \max_{\xi = 1}^\Xi \theta_\xi(\mu) \theta_\xi(\overline{\mu})^{-1}$:
\begin{align}
        \alpha(\mu, \overline{\mu})^{1/2} \enorm{\cdot}_{\overline{\mu}} \leq \enorm{\cdot}_\mu \leq \gamma(\mu, \overline{\mu})^{1/2} \enorm{\cdot}_{\overline{\mu}}.
        \label{eq:lrbms:norm_equivalence}
\end{align}
Note that since we consider an energy norm here, usage of the above norm equivalence requires no additional offline computations, in contrast to the standard min-$\theta$ approach \cite{PR2006}, where continuity and coercivity constants of $b(\cdot, \cdot; \overline{\mu})$ need to be computed when considering the $H^1_0$-norm.
We also denote by $c_\eps(\mu) > 0$ the minimum over $x \in \Omega$ of the smallest eigenvalue of the matrix $\lambda(x; \mu)\kappa_\eps(x) \in \R^{d \times d}$.

We are interested in a fully computable and offline/online decomposable estimate on the full approximation error in a fixed energy norm.
Therefore, we use the general framework presented in Section \ref{sec:general_framework} and apply it to the parametric setting of the localized RB method.
Since we use an implicit Euler time stepping for the reduced scheme we can readily apply Corollary \ref{thm:fully_discrete_estimate} and Example \ref{example:implicit_euler} by specifying all arising terms.

Given any discontinuous function $q_\red(\mu) \in Q_\red \subset Q_h^k \not\subset H^1_0(\Omega)$, we use the Oswald interpolation operator $I_\text{OS}: Q_h^k \to Q_h^k \cap H^1_0(\Omega)$, which consists of averaged evaluations of its source at Lagrange points of the grid $\grid$ (compare \cite[Section 4]{OS2015} and the references therein), to compute the conforming and non-conforming parts of a function by $q_\red^c := I_\text{OS}(q_\red)$ and $q_\red^d := q_\red - q_\red^c$, respectively.
Following Remark \ref{remark:elliptic_estimate}, we estimate the elliptic reconstruction error, $\enorm{\eps(n\cdot \Delta_t )}_\mu$, by the localizable and offline/online decomposable a posteriori error estimate $\eta(p_\red(n\cdot \Delta_t ; \mu); \mu, \mu, \tilde{\mu})$ from \cite[Corollary 4.5]{OS2015}, where $\tilde{\mu} \in \Params$ denotes any fixed parameter.

Since $b$ reduces on $H^1_0(\Omega)$ to the symmetric bilinear form \eqref{eq:lrbms:b_and_l}, we have $\enorm{b(\cdot, \cdot; \mu)}_{\mu} = 1$ for any $\mu \in \Params$.
Denoting the Poincar\'{e} constant with respect to $\Omega$ by $C_P^\Omega > 0$, we can estimate
$\norm{\Pi_\red(q)}_{L^2(\Omega)} \leq C_P^\Omega c_\eps(\mu)^{-1} \enorm{q}_\mu$ for any $\mu \in \Params$, $q \in
H^1_0(\Omega)$ and $\enorm{\partial_t p_\red^d(t)}_{\mu, Q, -1} \leq C_P^\Omega c_\eps(\mu)^{-1} \norm{\partial_t p_\red^d(t)}_{L^2(\Omega)}$ for $\partial_t p_\red^d \in L^2(0, T_\text{end}; Q_\red)$ and $\mu \in \Params$.
We thus obtain the following estimate by applying Corollary \ref{thm:fully_discrete_estimate} and Example \ref{example:implicit_euler} using the energy norm $\enorm{\cdot}_\mu$ and the norm equivalence \eqref{eq:lrbms:norm_equivalence}.

\begin{corollary}
\label{corollary:lrbms:estimate}
        Let the two partitions $\grid$ and $\Grid$ of $\Omega$ fulfill the requirements of \cite[Theorem 4.2]{OS2015}, namely: let $\grid$ be shape regular without hanging nodes and fine enough, such that all data functions can be assumed polynomial on each $t \in \grid$; let the subdomains $T \in \Grid$ be shaped, such that a local Poincar\'{e} inequality for functions in $H^1(T)$ with zero mean holds.
        For $\mu \in \Params$ let $p(\cdot; \mu) \in L^2(0, T_\text{end}; H^1_0(\Omega))$ denote the weak solution of the parabolic problem \eqref{eq:intro:continuous_problem} and let $p_\red(\cdot; \mu) \in L^2(0, T_\text{end}; Q_\red)$ denote the reduced solution of the fully-discrete problem \eqref{eq:lrbms:reduced}, where the constant function 1 is present in all local reduced bases spanning $Q_\red$.
        It then holds for arbitrary $\hat{\mu}, \overline{\mu}, \tilde{\mu} \in \Params$, that
        \begin{align*}
          \norm{p(\mu) - p_\red(\mu)}_{L^2(0, T_\text{end}; \enorm{\cdot}_{\overline{\mu}})}\\
                        \leq \alpha(\mu, \overline{\mu})^{-1/2} \Big\{\hspace{10pt}
                          &\norm{e^c(0; \mu)}_{L^2(\Omega)} \;+\; \sqrt{5}\; \norm{p_\red^d(\mu)}_{L^2(0, T_\text{end}; \enorm{\cdot}_\mu)}\\
                               +\; &2\, \alpha(\mu, \hat{\mu})^{-1}\, C_{H,Q}^b(\hat{\mu})\; \norm{\partial_t p_\red^d(\mu)}_{L^2(0, T_\text{end}; L^2(\Omega))}\\
                               +\; &(\sqrt{5} + 1)\; \eta_\textnormal{ell.}(p_\red(\mu), \mu, \tilde{\mu})\\
                               +\; &2\, \alpha(\mu, \hat{\mu})^{-1}\, C_{H,Q}^b(\hat{\mu})\; \norm{\mathcal{R}_T(p_\red(\mu); \mu)}_{L^2(0, T_\text{end}; L^2(\Omega))}
                        \Big\}\\
                        &\hspace{-79pt} =: \eta_{h, \red}(p_\red(\mu); \mu, \hat{\mu}, \overline{\mu}, \tilde{\mu})
        \end{align*}
	with $C_{H,Q}^b(\hat{\mu}) = C_P^\Omega c_\eps(\hat{\mu})^{-1}$ and
        \begin{align*}
          \eta_\textnormal{ell.}(p_\red(\mu); \mu, \tilde{\mu})^2 := \frac{4\Delta_t }{3} \sum_{n = 0}^{n_t} \Big\{ &\eta_\textnormal{OS2015}(p_\red(n\cdot \Delta_t ; \mu); \mu, \mu, \tilde{\mu})\\
        &+ \sum_{e \in \faces} b_p^e(p_\red(n\cdot \Delta_t ; \mu), p_\red(n\cdot \Delta_t ; \mu); \mu) \Big\}
        \end{align*}
        where $\eta_\textnormal{OS2015}$ denotes the estimate $\eta$ from \cite[Corollary 4.5]{OS2015}.
\end{corollary}

In Corollary \ref{corollary:lrbms:estimate}, we have the flexibility to choose three parameters $\hat{\mu}, \overline{\mu}, \tilde{\mu} \in \Params$: the parameter $\overline{\mu}$ can be used to fix a norm throughout the computational process (for instance during the greedy basis generation), while the purpose of the parameters $\hat{\mu}$ and $\tilde{\mu}$ is to allow all quantities to be offline/online decomposable, cf. \cite{OS2015}.
The price to pay for this flexibility are the additional occurrences of $\alpha$, which are equal to $1$ in the nonparametric case or if the parameters coincide.

\section{Numerical Experiments}
\label{sec:numerical_example}

We consider \eqref{eq:intro:continuous_problem} on $\Omega = [0, 5] \times [0, 1]$, $T_\text{end} = 0.05$, with $p_0 = 0$ and the data functions $f$, $\kappa$ and $\lambda$ from the multiscale example in \cite[Section 6.1]{OS2015}: $\kappa_\eps$ is the highly heterogeneous permeability tensor used in the first model of the 10th SPE Comparative Solution Project\footnote{\url{http://www.spe.org/web/csp/index.html}}, $f$ models a source and two sinks and $\lambda(\mu) := 1 + (1 - \mu) \lambda_c$, where $\lambda_c$ models a high-conductivity channel.
The role of the parameter $\mu \in \Params := [0.1, 1]$ is thus to toggle the existence of the channel, the maximum contrast of $\lambda(\mu)\kappa_\eps$ amounts to $10^6$ (compare Figure \ref{figure:data_functions}).

\begin{figure}[t]
  \centering%
  \footnotesize%
  \includegraphics{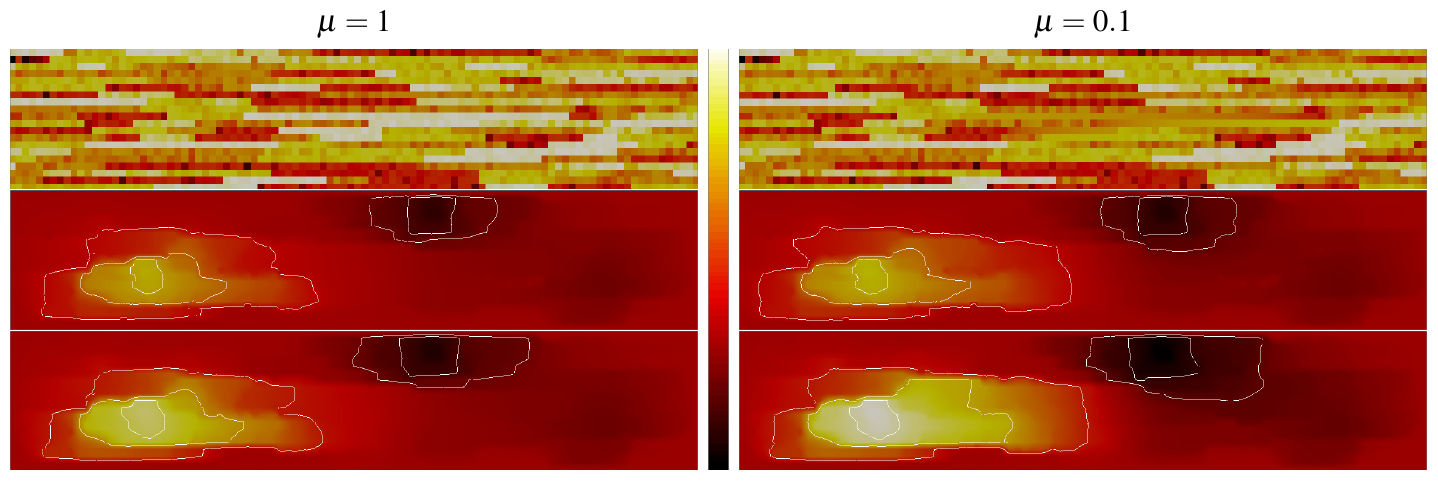}
  \caption{%
    Data functions and sample solutions on a grid with $|\grid| = 8,000$ simplices for parameters $\mu = 1$ (left column) and $\mu = 0.1$ (right column).
    Both plots in the first row as well as the bottom four plots share the same color map (middle) with two different ranges.
    First row: logarithmic plot of $\lambda(\mu)\kappa_\eps$ (dark: $1.41\cdot 10^{-3}$, light: $1.41\cdot 10^3$).
    Rest: plot of the pressure $p_h(t;\mu)$ (solution of \eqref{eq:reduced_problem}, dark: $-3.92\cdot 10^{-1}$, light: $7.61\cdot 10^{-1}$, isolines at 10\%, 20\%, 45\%, 75\% and 95\%) for $t = 0.01$ (middle row) and the end time $t = T_\text{end} = 0.05$ (bottom row).
    Note the presence of high-conductivity channels in the permeability (top left, light regions) throughout large parts of the domain.
    The parameter dependency models a removal of one such channel in the middle right of the domain.
    }
    \label{figure:data_functions}
\end{figure}

\noindent\textbf{Basis generation.}
On each subdomain $T \in \Grid$ we initialize the local reduced basis with $\varphi_\red^T := $\texttt{gram\_schmidt(}$\{ 1, \restrict{f}{T}\}$\texttt{)}, where \texttt{gram\_schmidt} denotes the Gram-Schmidt orthonormalization procedure (including re-orthonormalization for numerical stability) with respect to the full $H^1(T)$ product from our software package \pymor (see below).
The constant function $1$ has to be present in the local reduced bases according to \cite[Theorem 4.2]{OS2015} (to guarantee local mass conservation w.r.t.~subdomains), while the presence of $f$ sharpens the a posteriori estimate by minimizing $\Pi_\red(f) - f$, as motivated by the elliptic reconstruction \eqref{eq:elliptic_reconstruction}.
We iteratively extend these initial bases using a variant of the \textsc{POD-Greedy} algorithm \cite{HO2008}: in each iteration (i) the worst approximated parameter, say $\mu_* \in \Params_\text{train}$, is found by evaluating the a posteriori error estimate from Corollary \ref{corollary:lrbms:estimate} over a set of training parameters $\Params_\text{train} \subset \Params$; (ii) a full solution trajectory $\{p_h(n\cdot \Delta_t ; \mu_*) \,|\, 0 \leq n \leq n_t\}$ is computed using the discretization from Section \ref{section:lrbms}; and (iii) the local reduced bases $\varphi_\red^T$ for each subdomain $T \in \Grid$ are extended by the dominant POD mode of the projection error of $\{\restrict{p_h(n\cdot \Delta_t ; \mu_*)}{T} \,|\, 0 \leq n \leq n_t\}$, using the above Gram-Schmidt procedure.

\noindent\textbf{Software implementation.}
We use the open-source Python software package \pymor\footnote{\url{http://pymor.org}} \cite{MRS2015} for all model reduction algorithms as well as for the time stepping.
For the grids, operators, products and functionals we use the open-source \Cpp software package \DUNE, in particular the generic discretization toolbox \texttt{dune-gdt}\footnote{\url{http://github.com/dune-community/dune-gdt}} (see \cite[Section 6]{OS2015} and the references therein), compiled into a Python module to be directly usable in \pymor's algorithms.

We use a simplicial triangulation for the fine grid $\grid$, rectangular subdomains $T \in \Grid$ and $10$ equally sized time steps for the implicit Euler scheme.
Within each subdomain we use a local DG space of order $1$, the resulting discretization thus coincides with the one proposed in \cite{ESZ2009}.

\begin{SCfigure}[10]
  \footnotesize%
  \centering%
  \includegraphics{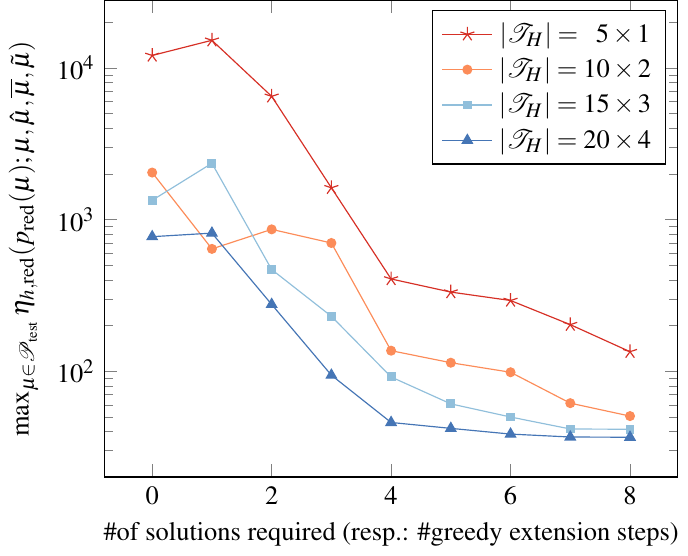}
  \caption{%
        Estimated error evolution during the \textsc{POD-Greedy} basis generation for several subdomain configurations and $\hat{\mu} = \overline{\mu} = \tilde{\mu} = 0.1$, to minimize all occurrences of $\alpha$ in Corollary \ref{corollary:lrbms:estimate}.
        Depicted is the maximum estimated error over a set of ten randomly chosen test parameters $\Params_\textnormal{test} \subset \Params$ in each step of the greedy algorithm, which was configured to search over ten uniformly distributed training parameters.%
  }
  \label{figure:greedy}
\end{SCfigure}

We observe a comparable decay of the estimated error during the greedy basis generation in Figure \ref{figure:greedy} for all subdomain configurations, though faster for a larger number of subdomains $|\Grid|$, where the reduced space is much richer.
In particular, to reach the same prescribed error tolerance in the greedy algorithm, much less solution snapshots are required for larger numbers of subdomains.
We refer to \cite[Section 3.3]{ORS2016} for a comparison of localized RB methods versus traditional RB methods.

\section{Conclusion}
\label{sec:conclusion}
In this contribution we used the elliptic reconstruction technique for a posteriori error estimation of parabolic problems \cite{MakridakisNochetto2003,LakkisMakridakis2006,DemlowLakkisEtAl2009,GeorgoulisLakkisEtAl2011} to derive efficient and reliable true error control for the localized reduced basis method applied to scalar linear parabolic problems.
Numerical experiments were given to demonstrate the applicability of the approach.

{\small%

}

\end{document}